\documentclass[amssymb,amsfonts,refcheck,12pt,verbatim,righttag]{amsart}
\usepackage{color}
\usepackage{amssymb}
\usepackage{graphics}
\usepackage{tikz}
\usepackage{multicol}
\usepackage{subfigure}

 \numberwithin{equation}{section} 
\newcommand{\beq}{\begin{equation}}
\newcommand{\eeq}{\end{equation}}

\newcommand{\ben}{\begin{eqnarray}}
\newcommand{\een}{\end{eqnarray}}

\newcommand{\bet}{\begin{eqnarray*}}
\newcommand{\eet}{\end{eqnarray*}}

\newtheorem{thm}{Theorem}[section]
\newtheorem{lem}[thm]{Lemma}

\newtheorem{prop}[thm]{Proposition}

\newtheorem{exa}[thm]{Example}

\usepackage[colorlinks, linkcolor=blue,citecolor=blue, anchorcolor=blue,urlcolor=blue,pagebackref,hypertexnames=false]{hyperref}

\newcommand{\R}{\mathbb{R}}

\newcommand{\Z}{\mathbb{Z}}

\theoremstyle{plain}

\begin{document}
\baselineskip 16pt

\title{On arithmetic sums of connected sets in $\R^2$}

\author{Yu-feng Wu}
\address[]{School of Mathematics and Statistics,\\ Central South University,\\ Changsha,
410085, PR China}\email{yufengwu.wu@gmail.com}



\begin{abstract}
We prove that for two connected sets $E,F\subset\R^2$ with cardinalities greater than $1$, if one of $E$ and $F$ is compact and not a line segment, then the arithmetic sum $E+F$ has non-empty interior. This improves a recent result of Banakh, Jabłońska and Jabłoński \cite[Theorem 4]{BJJ2019} in  dimension two by relaxing their assumption that $E$ and $F$ are both compact. 
\end{abstract}

\maketitle

\setcounter{section}{0}

\section{Introduction}

Given finitely many sets  $E_1,\ldots, E_n\subset \R^d$, their arithmetic sum is defined as 
\[E_1+\cdots+E_n=\{x_1+\cdots+x_n: x_i\in E_i \text{ for } 1\leq i\leq n\}.\]
A fundamental question is to find suitable conditions on $E_1,\ldots, E_n$ under which their arithmetic sum has non-empty interior. There are two classical results on this question. First, if two sets $E,F\subset \R^d$ are large in the sense of having positive Lebesgue measure (and measurable), then the (generalized) Steinhaus theorem states that $E+F$ has non-empty interior. Second, Piccard's thoerem says that the same conclusion holds when $E$ and $F$ are large in the sense of being of second category in $\R^d$ and having the Baire property. For a detailed account on these two results, the reader is referred to the monograph of Oxtoby \cite{Oxtoby80}. 

There also have been many works on the above question for sets which are small  in the sense of both measure and topology, which are often fractal sets. Studies in this direction also date back to a work of  Steinhaus, who in \cite{Steinhaus17} first observed that the arithmetic sum of the middle-third Cantor set with itself is the interval $[0,2]$. Subsequential generalizations include the works of Hall \cite{Hall47}, Newhouse \cite{Newhouse79} and Astels \cite{Astels00}, to merely name a few. Very recently, Feng and the author \cite{FW} studied the above question for fractal sets in higher dimensional Euclidean spaces. 

Recently, considerable attention has been given to the study of arithmetic sums involving connected sets. Chang \cite{AChang18} proved that if $E\subset \R^2$ is a curve connecting the points $(0,0)$ and $(1,0)$, and $F\subset \R^2$ is a curve connecting  $(0,0)$ and $(0,1)$, then $E+F+\Z^2=\R^2$. Simon and Taylor \cite{STinterior} studied the question when the arithmetic sum of a $C^2$ curve and a certain class of fractal sets in $\R^2$ has non-empty interior. They  \cite{STdimmea}  also studied the dimension and measure  of the arithmetic sum of a $C^2$ curve and a set in the plane. Very recently,  Banakh, Jabłońska and Jabłoński  proved the following result, which motivated the present paper. 

\begin{thm}\cite[Theorem 4]{BJJ2019}
\label{thm-BJJ}
Let $K_1,\ldots, K_d\subset \R^d$ be compact connected sets. Suppose that there exist $a_i,b_i\in K_i$ for $i=1,\ldots,d$ such that the vectors $a_1-b_1,\ldots, a_d-b_d$ are linearly independent. Then $K_1+\cdots +K_d$ has non-empty interior.
\end{thm}

Banakh {\it et al.} proved Theorem \ref{thm-BJJ} by  making an elegant use  of a result in topology  on products of continua (see \cite[Proposition 1]{BJJ2019}). A natural question arises to what extent   the compactness assumption in Theorem \ref{thm-BJJ} can be relaxed. In this note, we investigate this question in the case when $d=2$. By using a completely different approach, we  obtain the following result.  
\begin{thm}
\label{conthm-1.1}
Let $E, F\subset \R^2$ be  connected  sets with cardinalities greater than $1$. If  $F$ is compact and not a line segment,  then  $E+F$ has non-empty  interior.
\end{thm}

Theorem \ref{conthm-1.1} improves Theorem \ref{thm-BJJ} when $d=2$, since we allow one of the two connected sets to be non-compact. We also show that the assumptions in Theorem \ref{conthm-1.1} can not be further relaxed. More precisely, we show that  if $F\subset \R^2$ is a line segment, then  there exists a   non-compact connected set $E\subset \R^2$ not lying  in a line such that $E+F$ has empty interior. Also, we will give examples of non-compact connected sets $E,F\subset \R^2$, neither of which is contained in a line,  such that $E+F$  has empty interior. Theorem \ref{conthm-1.1} combining with these examples gives a  full answer to the above question on the compactness assumption in Theorem  \ref{thm-BJJ} in the case when $d=2$.

Our strategy to prove Theorem  \ref{conthm-1.1} is as follows: first, we prove the conclusion when the complement of $F$ has at least one bounded connected component. Then we refine this result by proving that if there is a compact set $K$ lying in a line in $\R^2$ such that the complement of $F\cup K$ has at least one bounded connected component, then $E+F$ has non-empty interior as well.  These two results are stated and proved in $\R^d$, see Lemmas \ref{conlem-1.2}-\ref{conlem-1.3}. Next, we prove that when $F$ has empty interior, there does exist a compact set $K$ lying in a line in $\R^2$ such that the complement of $F\cup K$ has at least one bounded connected component, see Proposition \ref{pro-1.4}. Finally, Theorem \ref{conthm-1.1} follows by combining Lemma \ref{conlem-1.3} and Proposition \ref{pro-1.4}.

The paper is organized as follows. In Section \ref{section2},  we give some preliminary lemmas. Then we prove Theorem \ref{conthm-1.1} in Section \ref{secconnected}. Finally, in Section \ref{secexamples} we present examples to show that the assumptions in Theorem \ref{conthm-1.1} can not be further relaxed.

\section{Preliminary lemmas}\label{section2}
For  $A\subset \R^d$, let $A^c$, $A^o$, $\partial A$ and $\overline{A}$  denote respectively the complement,  interior,  boundary and closure of $A$.  We first prove a useful lemma.

\begin{lem}
\label{conlem-1.2}
Let $E\subset\R^d$ be a  connected set with cardinality greater than $1$. Let $F\subset \R^d$ be a compact set so that $F^{c}$ has at least one bounded connected component. Then the arithmetic sum $E+F$ has non-empty interior.
\end{lem}
\begin{proof}
We assume that the interior of $F$ is empty, otherwise there is nothing left to prove. Let $U$ be the unbounded connected component of $F^c$. Write $V=F^c\backslash U$. By the assumption that $F^c$ has at least one bounded connected component, $V$ is a non-empty bounded open subset of $\R^d$. Set
\begin{equation}
\label{e-0}W=\{x\in \R^d:\; \exists a, b\in E  \mbox{ such that }x\in (a+U)\cap (b+V)\}.
\end{equation}
Clearly $W$ is open. We claim that $W\neq\emptyset$. To prove this claim, it suffices to show that $(s+V)\cap U\neq \emptyset$ for each non-zero $s\in \R^d$.  To this end, fix a
non-zero $s\in \R^d$ and define $P_s:\R^d\to \R$ by
$P_s(x)=\langle x,s \rangle$, where $\langle\cdot , \cdot\rangle$ is the standard inner product in $\R^d$. Since $V$ is bounded, $$\lambda:=\sup_{x\in V}P_s(x)<\infty.$$
Pick $x_0\in V$ so that $P_s(x_0)>\lambda-\|s\|^2/2$, and take a small $r\in (0, \|s\|/2)$ so that $B^{o}(x_0, r)\subset V$, where $B^{o}(x_0, r)$ stands for the open ball centered at $x_0$ of radius $r$. Then for each $y\in \R^d$ with $\|y-x_0\|<r$,
\begin{eqnarray*}
\langle s+y,s\rangle &=& \langle s+x_0,s\rangle +\langle y-x_0,s\rangle\\
&\geq &  \langle x_0,s\rangle +\|s\|^2 -\|y-x_0\|\cdot \|s\|\\
&\geq &  \langle x_0,s\rangle +\|s\|^2/2\\
&>& \lambda,
\end{eqnarray*}
which implies that $s+y\not\in V$, and so $s+B^{o}(x_0, r)\subset V^c$.  Since $F$ has empty interior, it follows that $(s+B^{o}(x_0, r))\cap F^c\neq \emptyset$, equivalently, $(s+B^{o}(x_0, r))\cap (U\cup V)\neq \emptyset$. Since $s+B^{o}(x_0, r)\subset V^c$, we get $(s+B^{o}(x_0, r))\cap U\neq \emptyset$, so $(s+V)\cap U\neq \emptyset$, as desired. This proves $W\neq \emptyset$.

Finally, we prove that $W\subset E+F$, which immediately implies that $E+F$ has non-empty interior. Suppose this is not true, i.e., there exists $x\in W$ so that $x\not\in E+F$.  By the definition of $W$, there exist $a, b\in E$ so that
\begin{equation}
\label{e-1} x\in (a+U)\cap (b+V).
\end{equation}
Since $x\not\in E+F$, it follows that 
\begin{equation}
\label{e-to1}
(x-E)\subset F^c=U\cup V.
 \end{equation}
 However, according to \eqref{e-1}, $(x-E)\cap U\supset \{x-a\}$ and  $(x-E)\cap V\supset \{x-b\}$. This together with \eqref{e-to1} implies that $x-E$ is not connected, leading to a contradiction.
\end{proof}
\bigskip

The next lemma is a refined version of Lemma \ref{conlem-1.2}.

\begin{lem}
\label{conlem-1.3}
Let $E\subset \R^d$ be a connected set with cardinality greater than $1$. Let $F\subset \R^d$ be compact. Suppose that there exists a compact set $K\subset \R^d$ so that the following two properties hold:
\begin{itemize}
\item[(i)] $K$ is contained in a hyperplane.
\item[(ii)] $(F\cup K)^c$ has at least one bounded connected component.
\end{itemize}
Then the arithmetic sum $E+F$ has non-empty interior.
\end{lem}

Lemma \ref{conlem-1.3} is a direct consequence of the following result.

\begin{lem}\label{conlem-1.4}
Let $F, K\subset \R^d$ be as in Lemma \ref{conlem-1.3}. Let $D$ be a dense subset of $\R^d$. Then the set 
\[A:=\bigcap_{z\in D}(z-F)^c\]
is totally disconnected.
\end{lem}

\begin{proof}
By taking a suitable rotation and translation if necessary, we may assume that $K$ is contained in the linear subspace $\{(x_1,\ldots,x_d)\in \R^d: x_d=0\}$, and that at least one  bounded connected component of $(F\cup K)^c$  has non-empty intersection with the half-space $H:=\{(x_1,\ldots, x_d)\in \R^d: x_d>0\}$.

Let $U$ be a bounded connected component of $(F\cup K)^c$ so that 
\[V:=U\cap H\]
is non-empty. Since  $\partial U\subset F\cup K$ and  $K\subset \partial H$, we easily see that  
\begin{align*}
\partial V\subset F\cup\left(\overline{U}\cap \partial H\right).
\end{align*}
Write $\tilde{K}=\overline{U}\cap \partial H$. Then  $\tilde{K}$ is a compact subset of the linear subspace $\partial H=\{(x_1,\ldots,x_d)\in\R^d: x_d=0\}$ and 
\begin{equation}\label{eqparV}
\partial V\subset F\cup \tilde{K}.
\end{equation}
Set $$h:=\sup\Pi(V),$$
 where $\Pi: \R^d\to\R$ is  the mapping $(x_1, \ldots, x_d)\mapsto x_d$. Then $h$ is positive and finite.

To prove that  $A$ is totally disconnected,  we may assume that $\#A\geq 2$, since otherwise there is  nothing to prove.  
Let $u,v\in A$ with $u\neq v$. In the following, we are going to construct an open set $W\subset \R^d$ such that 
\begin{equation}\label{eq1211423}
u\in W, \  v\not\in \overline{W} \ \text{ and } \ \partial W\cap A=\emptyset. 
\end{equation}
Since $u,v\in A$ with $u\neq v$ are arbitrary,  it will follow  that $A$ is totally disconnected. 

To prove \eqref{eq1211423}, first notice that  the open set  $(u+V)\setminus (v+\overline{V})$ is  non-empty. Indeed, since $u\neq v$ and  $(u+V)\setminus (v+\overline{V})=\left(\left(u-v+V\right)\setminus \overline{V}\right)+v$, it suffices to show that  $(a+V)\setminus \overline{V}\neq \emptyset$ for any  non-zero $a\in \R^d$.  To this end, fix  $a\in \R^d$ with $a\neq 0$.  Define $$\lambda:=\sup_{x\in V}\langle x, a\rangle.$$
where $\langle\cdot, \cdot\rangle$ is the standard inner product in $\R^d$. Then  $\lambda$ is finite as  $V$ is bounded. Also, it is clear that $\lambda=\sup_{x\in\overline{V}}\langle x, a\rangle$.   Since $a\neq 0$, we have $\langle a, a\rangle>0$. Hence we can find  $x_0\in V$  such that 
\[\langle x_0, a\rangle>\lambda-\langle a,a\rangle.\]
Thus we have
\[\langle a+x_0,a\rangle=\langle x_0,a\rangle+\langle a,a\rangle>\lambda,\]
which implies that  $a+x_0\in (a+V)\setminus \overline{V}$. Hence we see that $(u+V)\setminus (v+\overline{V})$ is a non-empty open set. By the density of  $D$, we can pick a point $z\in D\cap \left[(u+V)\setminus (v+\overline{V})\right]$. Then we have  
\begin{equation}\label{equv}
u\in z-V \ \text{ and } \  v\notin z-\overline{V}.
\end{equation}

Let \[\tilde{D}=\left\{x\in D: \Pi(x)>\Pi(u)+h\right\}.\]
Since $D$ is dense in $\R^d$, $\tilde{D}$ is clearly dense in  the open half-space
\begin{equation}\label{eq1211452}
\left\{x\in \R^d: \Pi(x)>\Pi(u)+h\right\}.
\end{equation} 
 Let $x\in \R^d$ with $\Pi(x)>\Pi(u)$. Since $x+V$ is open and
\[\sup\Pi(x+V)=\Pi(x)+h>\Pi(u)+h,\]
we  see from \eqref{eq1211452} that   $(x+V)\cap \tilde{D}\neq \emptyset$. Equivalently, $x\in \tilde{D}-V$. Hence we have 
\begin{equation}\label{eq:strip}
     \tilde{D}-V\supset\left\{x\in \R^d: \Pi(x)>\Pi(u)\right\}.
\end{equation}

Since $\tilde{K}$ is contained in the linear subspace $\{(x_1,\ldots, x_d)\in \R^d: x_d=0\}$,   $V$ is contained in the open half-space $\{(x_1,\ldots, x_d)\in \R^d: x_d>0\}$ and $z\in u+V$, we have 
\begin{equation}\label{eq:Pi2}
 \Pi(z-\tilde{K})=\{\Pi(z)\} \ \text{ and } \ \Pi(z)>\Pi(u). 
\end{equation}
Then by \eqref{eq:strip}, \eqref{eq:Pi2} and the compactness of $z-\tilde{K}$,  we can find finitely many points $z_1,\ldots, z_k\in \tilde{D}$ so that 
\begin{equation}\label{eq:z-L}
z-\tilde{K}\subset \bigcup_{i=1}^k(z_i-V).
\end{equation}
Set 
\begin{equation}\label{eq:def U}
W:=(z-V)\setminus \left(\bigcup_{i=1}^k \left(z_i-\overline{V}\right)\right).
\end{equation}
See Figure \ref{figV} for an illustration of the definition of $W$, where for simplicity we assume that $V$ is  an open half disk. 
Below we show that the open set $W$ satisfies \eqref{eq1211423}.

\begin{figure}
\centering
\begin{tikzpicture}
\draw[->] (0,-1)--(0,5);
\draw[->] (-1.5,0)--(5,0);
\draw  (0,0)--(0:1) arc (0:180:1)--cycle; 
\draw  (2,2.5) arc [radius=1, start angle=180, end angle=360]--cycle;
\draw  (1.5,3) arc [radius=1, start angle=180, end angle=360]--cycle;
\draw  (2,3.3) arc [radius=1, start angle=180, end angle=360]--cycle;
\draw  (2.5,2.9) arc [radius=1, start angle=180, end angle=360]--cycle;
\draw [ultra thick] (-1,0)--(1,0);
\draw [ultra thick] (2,2.5)--(4,2.5);
\draw [ultra thick] (1.5,3)--(3.5,3);
\draw [ultra thick] (2,3.3)--(4,3.3);
\draw [ultra thick] (2.5,2.9)--(4.5,2.9);
\draw[fill] (2.65,1.8) circle [radius=0.025];
\draw[fill] (4.5,2) circle [radius=0.025];

\node [left] at (0, 5) {$y$};
\node [below] at (5, 0) {$x$};
\node  at (-0.5, 0.5) {$V$};
\node[below]  at (-0.5, 0) {$\tilde{K}$};
\node[left]  at (2.73, 1.8) {$u$};
\node[below] at (4.5, 2) {$v$};
\node[right]  at (2.8, 1.7) {$W$};

\node[below]  at (3, 1.5) {$z-V$};
\node[left]  at (1.5, 3) {$z_1-V$};
\node[above]  at (3, 3.3) {$z_2-V$};
\node[below] at (5.1, 2.8) {$z_3-V$};

\end{tikzpicture}
\caption{}
\label{figV}
\end{figure}

First notice that $v\notin \overline{W}$ since $v\notin z-\overline{V}$; see \eqref{equv}. Moreover,
since $z_1,\ldots,z_k\in \tilde{D}$, we have 
\[\inf\Pi\left(\bigcup_{i=1}^k\left(z_i-\overline{V}\right)\right)=\min_{1\leq i\leq k}\left(\Pi(z_i)-h\right)>\Pi(u),\]
which implies that 
$$u\notin \bigcup_{i=1}^k\left(z_i-\overline{V}\right).$$
Since $u\in z-V$, it follows that $u\in W$.  In the following, we show that $\partial W\cap A=\emptyset$.

To see this,  observe that by \eqref{eqparV}  and  \eqref{eq:def U} we have 
\begin{equation}\label{eq:partical U}
 \partial W\subset (z-F)\cup (z-\tilde{K})\cup\left(\bigcup_{i=1}^k(z_i-F)\right)\cup\left(\bigcup_{i=1}^k(z_i-\tilde{K})\right).  
\end{equation}
By \eqref{eq:z-L}, \eqref{eq:def U} and the compactness of $z-\tilde{K}$, we see that $W$ is disjoint from a neighborhood of $z-\tilde{K}$. In particular, this implies that 
\begin{equation}\label{eq: U cap z-L}
    \partial W\cap (z-\tilde{K})=\emptyset.
\end{equation}
 Next we show that 
\begin{equation}\label{eq: partical U cap}
    \partial W\cap \left(\bigcup_{i=1}^k (z_i-\tilde{K})\right)=\emptyset.
\end{equation}
To see this,   since $\tilde{K}\subset \{(x_1,\ldots,x_d)\in \R^d: x_d=0\}$, we have for $i\in \{1,\ldots,k\}$,
\begin{equation}\label{eq:Pi2(zi-L)}
\Pi(z_i-\tilde{K})=\{\Pi(z_i)\}.
\end{equation}
Moreover, since $z_1,\ldots, z_k\in \tilde{D}$, we see that  
\begin{equation}\label{eq1211633}
\Pi(z_i)>h+\Pi(u), \ \ i=1,\ldots,  k.
\end{equation}
On the other hand,  since $V\subset \{(x_1,\ldots,x_d)\in\R^d: x_d>0\}$  and  $z\in u+V$,  
\begin{equation}\label{eq:h+Pi2(u)}
    h+\Pi(u)\geq \Pi(z)\geq \sup\Pi(z-\overline{V})\geq \sup\Pi\left(\overline{W}\right),
\end{equation}
where the last inequality is by \eqref{eq:def U}. 
Now \eqref{eq:Pi2(zi-L)}-\eqref{eq:h+Pi2(u)} imply that for  $i\in \{1,\ldots, k\}$, 
\[\inf\Pi(z_i-\tilde{K})=\Pi(z_i)>\sup\Pi(\partial W).\]
From this  \eqref{eq: partical U cap} follows.

 By \eqref{eq:partical U}-\eqref{eq: partical U cap} we see that 
 \begin{equation}\label{eqUsub}
 \partial W\subset (z-F)\cup\left(\bigcup_{i=1}^k(z_i-F)\right).
 \end{equation}
 Since $z, z_1,\ldots, z_k\in D$, the definition of $A$ implies that $A$ has  no intersection with the right hand side of \eqref{eqUsub}. As a consequence, $A\cap \partial W=\emptyset$. Hence  \eqref{eq1211423} is proved and we finish the proof of the lemma. For an illustration of the proof, see Figure \ref{figV}. 
 \end{proof}

Now we deduce  Lemma \ref{conlem-1.3} from Lemma \ref{conlem-1.4}.
\begin{proof}[Proof of Lemma \ref{conlem-1.3}]
We prove the lemma by contradiction. Suppose that $(E+F)^{\circ}=\emptyset$. Then we have, equivalently,  that the set \[D:=(E+F)^c\]
is dense in $\R^d$.  Hence by Lemma \ref{conlem-1.4},  $\bigcap_{z\in D}(z-F)^c$ is totally disconnected. However, from the definition of $D$ we see that 
\begin{equation}\label{eqEsubset}
E\subset \bigcap_{z\in D}(z-F)^c.
\end{equation}
This contradicts the assumption that $E$ is a connected set with cardinality greater than $1$. Hence we have  $(E+F)^{\circ}\neq \emptyset$, completing the  proof of the lemma. 
\end{proof}

\section{Proof of Theorem \ref{conthm-1.1}}\label{secconnected}

We first state a classical result in convex analysis. The reader is referred to \cite[Theorem 17.1]{Rockafellar1970} for a proof.

\begin{thm}[Carath\'{e}odory's Theorem] Let $S\subset \R^d$ and let ${\rm conv}(S)$ denote the convex hull of $S$. Then any $x\in {\rm conv}(S)$ can be represented as a convex combination of $d+1$ elements of $S$.
\end{thm}

Now we apply the above theorem to prove the following  fact.
\begin{lem}
\label{lem-1.4}
Let $S$ be a compact subset of $\R^2$.  Suppose that $x\in \partial ({\rm conv}(S))\backslash S$. Then there exist $u, v\in S$ with $u\neq v$ such that $x$ is contained in the line segment connecting $u,v$, and moreover, $S$ lies completely on one side of the line passing through $u, v$.
\end{lem}
\begin{proof}
The result might be  well-known. However we are not able to find a reference, so we simply include a proof.  By Carath\'{e}odory's Theorem, $x$ can be represented as a convex combination of $3$ elements of $S$. Equivalently, $x$ lies in a triangle with vertices in $S$.  Since $x$ is on the boundary of  ${\rm conv}(S)$,  it follows that $x$ lies on one edge of the triangle. Let $u,v\in S$ be the endpoints of this edge. Since $x\not\in S$, we have $u\neq v$.

Let $L_{u,v}$ denote the straight line passing through the points $u, v$. We show that $S$ lies completely on one side of $L_{u, v}$. Suppose on the contrary that $S$ does not lie on one side of  $L_{u,v}$. Then we can pick $w_1, w_2\in S$ such that $w_1, w_2$ lie on different sides of $L_{u,v}$. Clearly, the point $x$ lies in the interior of the quadrilateral with vertices $u,  v, w_1, w_2$ (see Figure \ref{fig:1}). However, this quadrilateral is a subset of ${\rm conv}(S)$, contradicting  the assumption that $x\in \partial ({\rm conv}(S))$.
\end{proof}

\begin{figure}
    \centering
    \begin{tikzpicture}[scale=1.5]
    \draw (0,0)--(6,3);
    \draw[fill] (4,2) circle [radius=0.025];
    \node [below] at (4,2) {$x$};
        \draw[fill] (4,2) circle [radius=0.025];
    \draw[fill] (5,2.5) circle [radius=0.025];
    \node [right,below] at (5,2.5) {$v$};
    \draw[fill] (1.6,0.8) circle [radius=0.025];
    \node [below] at (1.6,0.8) {$u$};
    \draw[fill] (1.8,2.8) circle [radius=0.025];
    \node [left] at (1.8,2.8) {$w_1$};
    \draw[fill] (4.4,1.1) circle [radius=0.025];
    \node [right] at (4.4,1.1) {$w_2$};
    \draw (5,2.5)--(1.8,2.8);
    \draw (1.6,0.8)--(1.8,2.8);
    \draw  (1.6,0.8)--(4.4,1.1);
    \draw (4.4,1.1)--(5,2.5);
    \node[above] at (0.1,0.1){$L_{u,v}$};
    \end{tikzpicture}
    \caption{}
    \label{fig:1}
\end{figure}

Next we  prove the following proposition, which plays a key role in the proof of Theorem \ref{conthm-1.1}.

\begin{prop}
\label{pro-1.4}
Let $F$ be a compact connected subset of $\R^2$ with empty interior. Suppose that $F$ is not lying in a line. Then there exists a compact set $K\subset \R^2$ lying in a line  such that  $(F\cup K)^c$ has at least one bounded connected component.
\end{prop} 
\begin{proof}
We may assume that $F^c$ has no bounded connected components, otherwise we simply take $K=\emptyset$.

Let ${\rm conv}(F)$ denote the convex hull of $F$. Since $F$ is not contained in  a straight line, ${\rm conv}(F)$ has non-empty interior and there exists a homeomorphism $h:\R^2\to \R^2$ so that $h({\rm conv}(F))$ is the unit closed ball centered at the origin (see e.g. \cite[Exercise 8.11]{BorweinLewis2006}).

We claim that $\partial ({\rm conv}(F))\not\subset F$. Suppose on the contrary that   $\partial ({\rm conv}(F))\subset F$. As $\partial ({\rm conv}(F))$ is homeomorphic to the unit circle, $\R^2\backslash \partial ({\rm conv}(F))$ has exact two connected components $V_1, V_2$ where $V_1$ is unbounded and $V_2$ bounded. Since $\partial ({\rm conv}(F))\subset F$, $F^c\subset V_1\cup V_2$. Since $F$ is compact, $F^c\cap V_1\neq \emptyset$, and so $F^c\subset V_1$ by the connectedness of $F^c$. This implies that $V_2\subset F$, contradicting the assumption that $F$ has empty interior. 

Pick $z\in \partial ({\rm conv}(F))\backslash F$. By Lemma \ref{lem-1.4}, there exist $u,v\in F$ such that the straight line $L_{u,v}$ passes through $z$ and $F$ lies completely on one side of $L_{u,v}$.   By taking suitable rotation and translation to $F$, we may assume that $z=(0,0)$, $L_{u,v}$ is the $x$-axis,  $u$ is on the negative part of the $x$-axis and $v$ the positive part of $x$-axis,  and $F$ lies entirely on the upper half plane.

Choose a large $R>0$ such that $F$ is contained in the closed half disc
$$
S:=\{(x, y)\in \R^2:\; y\geq 0,\; x^2+y^2\leq R^2\}.
$$
Let $K$ be the line segment $[-R, R]\times \{0\}$, which is the bottom edge of $S$. Below we show that $(F\cup K)^c$ has at least one bounded connected component.

Since the origin is not contained in $F$, there exists a small $r>0$ such that the following open half disc
$$
T:=\{(x, y)\in \R^2:\; y> 0,\; x^2+y^2< r^2\}
$$
is contained in $(F\cup K)^c$.  Let $V$ be the connected component of $(F\cup K)^c$ that contains $T$. Notice that $S^c$ is contained in the unbounded connected component $U$ of $(F\cup K)^c$. To show that $V$ is bounded, it is enough to show that $V\neq U$ (keep in mind that $(F\cup K)^c$ has a unique unbounded connected component, due to the compactness of  $F\cup K$).

 Suppose on the contrary that $V=U$. Then $U\supset T\cup S^c$. Pick $a\in T$ and $b\in S^c$. Since $U$ is open and connected, there exists a simple curve $\gamma\subset U$ such that $\gamma$ consists of finitely many line segments and $\gamma$ joins the points $a, b$ (see e.g. \cite[p.~56]{Ahlfors1978} for a proof). Clearly,  $\gamma$ must intersect  the open half circle $$\Gamma:=\{(x,y):\; x^2+y^2=R^2,\; y>0\}$$
  at one or  more than one  points. As $\gamma$ is a polygon, we may choose a sub-polygon $\gamma_1$ which joins $a$ and a point $c \in \Gamma$ such that  $c$ is the unique intersection point of $\gamma_1$ and $\Gamma$. Connect the point $a$ and the origin by a simple polygon $\gamma_2\subset \overline{T}$ such that $\gamma_2$ intersects $\gamma_1$ only at the point $a$, and $\gamma_2$ intersects $K$ only at the origin.

 Let $\eta=\gamma_1\cup \gamma_2$. Then $\eta$ is a simple polygon, joining the origin and the point $c$. Except the endpoints, other points of $\eta$ are contained in $U\cap S^o$. Hence $\eta\cap F=\emptyset$.

 Write $c=(c_1, c_2)$. Let $L^+, L^{-}$ be the vertical half lines defined by $L^+:=\{(c_1, y): y\geq c_2\}$ and $L^-:=\{(0, y): y\leq 0\}$.  Then the union $\eta \cup L^+\cup L^-$ has no intersection with $F$, moreover its complement has two connected components, with $u, v$ being contained in different components. This implies that $F$ is disconnected, leading to a contradiction. See Figure \ref{fig:2} for an illustration of the proof. 
\end{proof}

\begin{figure}
    \centering
    \begin{tikzpicture}[scale=0.8]
    \draw (-4,0)--(4,0);
    \draw[fill] (0,0) circle [radius=0.025];
    \node [below] at (0.2,0.1) {$O$};
    \node [above] at (-0.3,0) {$T$};
     \node [above] at (-2.5,1) {$S$};
      \node [right] at (1.7,1.9) {$\eta$};

    \begin{scope}
        \clip (-4,0) rectangle (4,4);
    \draw (0,0) circle(4);
    \end{scope}
    \begin{scope}
        \clip (-0.9,0) rectangle (0.9,0.9);
    \draw (0,0) circle(0.9);
    \end{scope}
   
    \draw (0,0)--(0.2,0.4)--(0.4,0.3)--(0.8,0.7)--(1,1.6)--(1.6,1.8)--(2.2,2.5)--(2.4,3.2);
     \draw[red](0,-3)--(0,0)--(0.2,0.4)--(0.4,0.3)--(0.8,0.7)--(1,1.6)--(1.6,1.8)--(2.2,2.5)--(2.4,3.2)--(2.4,5.2);
    \draw (2.4,3.2)--(2.5,3.55)--(3,3.7)--(3.5, 4.2);
    \draw[fill] (3.5,4.2) circle [radius=0.025];
    \draw[fill] (0.4,0.3) circle [radius=0.025];
    \draw[fill] (2.4,3.2) circle [radius=0.025];
    \node[right] at (0.4,0.3) {$a$};
    \node[right] at (2.4,3.2) {$c$};
    \node[right] at (3.5, 4.2) {$b$};
    \node[left] at (0,-2.5) {$L^-$};
    \node[left] at (2.4,5) {$L^+$};
    \node[below] at (-2.5,0) {$u$};
    \node[below] at (1.5,0) {$v$};
    \draw[fill] (-2,0) circle [radius=0.025];
    \draw[fill] (1.5,0) circle [radius=0.025];
    \end{tikzpicture}
    
    \caption{}
    \label{fig:2}
\end{figure}

Now we combine Lemma \ref{conlem-1.3} and Proposition \ref{pro-1.4} to prove Theorem \ref{conthm-1.1}. 

\begin{proof}[Proof of Theorem \ref{conthm-1.1}] We can  assume that $F^{\circ}=\emptyset$, since otherwise there is nothing to prove. Since $F$ is compact, connected, $F^{\circ}=\emptyset$ and $F$ is not a line segment,   by   Proposition \ref{pro-1.4} there exists a compact set $K$ contained in a line segment such that $(F\cup K)^c$ has a bounded connected component. Then it follows from  Lemma \ref{conlem-1.3} that $E+F$ has non-empty interior.
\end{proof}

\section{Some examples}\label{secexamples}

We have proved our main result Theorem \ref{conthm-1.1} in the previous section: if $E,F\subset \R^2$ are  connected sets with cardinalities greater than $1$, and $F$ is compact and not a line segment, then $E+F$ has non-empty interior. 
In this section, we present examples to show that the assumptions in   this result can not be further relaxed.

Our first example shows that there are non-compact connected sets $E,F\subset \R^2$, neither of  which is contained in a line,  such that 
$E+F$ has empty interior. Therefore  the compactness assumption for $F$ in Theorem \ref{conthm-1.1} can not be dropped. 

The example that we will give involves a result on additive functions on $\R$. 
A function  $f: \R\to \R$ is said to be {\em additive} if $f(x+y)=f(x)+f(y)$ for all $x, y \in \R$. It is well-known that under some
regularity assumptions, for instance continuity at a point or Lebesgue measurability, an additive function is necessarily linear. However, F. B. Jones \cite[Theorem 5]{Jones1942} proved  the existence of discontinuous additive functions with connected graphs. 
Based on this result we give the following example.

\begin{exa}
Let $f:\R\to\R$ be a discontinuous additive function whose graph $G_f:=\left\{(x,f(x)): x\in\R\right\}$ is connected. Let $E=F=G_f$. Then both $E, F$ are connected and  not contained in a line in $\R^2$. Moreover, since $f$ is additive, it follows that  $E+F=G_f$ and so $E+F$ has empty interior. 
\end{exa}

We next give  examples to show that the conclusion of Theorem \ref{conthm-1.1} may fail if  $F\subset \R^2$ is a line segment, and $E\subset \R^2$ is a connected set which is not contained in a line in $\R^2$. In our examples, we will take $F$ to be a vertical line segment and $E$  the graph of a certain function. 

We first give a simple necessary and
sufficient condition in terms of the oscillations of a function $f:\R\to\R$ for the existence of a vertical
line segment $L$ such that $G_f+L$ has empty interior. 

Given a function $f:\R\to\R$, the  oscillation of $f$ at a point $x\in \R$ is defined by \[\omega_f(x)=\lim_{\delta\to0}\left[\sup_{y\in [x-\delta,x+\delta]}f(y)-\inf_{y\in [x-\delta,x+\delta]}f(y)\right].\]
We say that $f$ is {\em uniformly oscillated} if $\inf_{x\in \R}\omega_f(x)>0$. Clearly, $f$ is not uniformly oscillated if $f$ has a point of continuity. 
 
\begin{lem}\label{lem:iff con}
Let $f:\R\to\R$ be a function. Then there exists a vertical line segment $L\subset \R^2$ such that $G_f+L$ has empty interior if and only if $f$ is uniformly oscillated. 
\end{lem}

\begin{proof}
In one direction, assume that $\inf_{x\in \R}\omega_f(x)>0$. Let $L$ be a vertical line segment with length $0<\ell<\inf_{x\in\R}\omega_f(x)$. Below we show that $(G_f+L)^{\circ}=\emptyset$. 

By applying a suitable translation we can assume that $L=\{0\}\times[0,\ell]$. Suppose on the contrary that $(G_f+L)^{\circ}\neq\emptyset$. Then in particular $G_f+L$ contains a horizontal line segment, say, $[a,b]\times\{c\}$ for some $a,b,c\in \R$ with $a<b$. Notice that 
\[G_f+L=\bigcup_{x\in \R}\left((x,f(x))+L\right)=\bigcup_{x\in \R}\left(\{x\}\times[f(x), f(x)+\ell]\right)\]
is a disjoint union of vertical line segments. By this and our assumption that $G_f+L\supset [a,b]\times\{c\}$, we easily see that
\[(x,c)\in \{x\}\times[f(x), f(x)+\ell], \ \ \forall x\in [a,b],\]
and thus
\begin{equation}\label{eq:bdd f(x)}
c-\ell\leq f(x)\leq c, \ \ \forall x\in [a,b]. 
\end{equation}
Let $x_0=(a+b)/2$.  Then \eqref{eq:bdd f(x)} clearly implies that $\omega_f(x_0)\leq \ell$, contradicting that $\inf_{x\in\R}\omega_f(x)>\ell$.   The contradiction yields that $(G_f+L)^{\circ}=\emptyset$. 

In the other direction, we will prove that if $\inf_{x\in\R}\omega_f(x)=0$, then $(G_f+L)^{\circ}\neq\emptyset$ for any vertical line segment $L$. Assume that $\inf_{x\in\R}\omega_f(x)=0$. Let $L$ be a vertical line segment with length $\ell>0$. Again we can assume that $L=\{0\}\times[0,\ell]$.

Since $\inf_{x\in\R}\omega_f(x)=0$, then by definition we can find $x_0\in\R$ and $\delta>0$ such that 
\[\sup_{x\in [x_0-\delta, x_0+\delta]}f(x)-\inf_{x\in [x_0-\delta, x_0+\delta]}f(x)<\ell/4.\]
This clearly implies that
\begin{equation}\label{eq:f,l/4}
    f(x_0)-\ell/4<f(x)<f(x_0)+\ell/4
\end{equation}
for all $x\in [x_0-\delta, x_0+\delta]$. 
We claim that $G_f+L$ contains the rectangular \[R:=[x_0-\delta, x_0+\delta]\times[f(x_0)+\ell/4, f(x_0)+(3\ell)/4].\]
To see this, let $(x,y)\in R$. Then $f(x_0)+\ell/4\leq y\leq f(x_0)+(3\ell)/4$. By this and \eqref{eq:f,l/4} we see that $0\leq y-f(x)\leq \ell$. Hence $(0, y-f(x))\in L$. Therefore we have
\[(x,y)=(x, f(x))+(0,y-f(x))\in G_f+L.\]
Since $(x,y)\in R$ is arbitrary, it follows that $R\subset G_f+L$. This proves the above claim, and in particular, that $(G_f+L)^{\circ}\neq\emptyset$.
\end{proof}

According to Lemma \ref{lem:iff con}, if $f:\R\to\R$ is a uniformly oscillated function with a connected graph, then let $E$ be $G_f$ and $F$ be an appropriate vertical line segment, we have $E+F$ has non-empty interior. Such functions do exist as shown in the following examples. 

\begin{exa}\label{exa:jones}
F. B. Jones \cite[Theorems 1-2]{Jones1942} proved that there are  additive functions on $\R$  whose graphs are connected and dense in $\R^2$. Let $f$ be such a function. Since $G_f$ is dense in $\R^2$, it is easy to see that $\inf_{x\in\R}\omega_f(x)=\infty$.  Let $E=G_f$ and let  $F\subset \R^2$ be a vertical line segment. Then from the proof of Lemma \ref{lem:iff con} we see that  $E+F$ has empty interior. 
\end{exa}

Recently, Rosen \cite{rosen} proved that  the set  $E$ in Example \ref{exa:jones}  has positive two dimensional Lebesgue measure. Hence by  Fubini's theorem, $E$ is not Lebesgue measurable. Below we give another example in which $E$ is Borel. 

\begin{exa}
The Ces{\`a}ro function $f:[0,1]\to \R$ is defined by 
\[f(x):=\limsup_{n\to \infty}\frac{a_1(x)+\cdots+a_n(x)}{n},\]
where 
\[x=\sum_{n=1}^{\infty}\frac{a_n(x)}{2^n}\]
is the binary expansion of $x$. Here we adopt the convention that $a_n(x)=1$ for all large $n$ if $x$ has two different binary expansions. 

Notice that  $f$ is a Borel function, and hence its graph $G_f$ is a Borel subset of $\R^2$. Also,  it is easy to check that $\inf_{x\in[0,1]}\omega_f(x)=1$. Moreover,  Vietoris \cite{vietoris21} proved that $G_f$ is connected.

Let $E=G_f$ and  $F\subset \R^2$ be a vertical line segment of length less than $1$. Then we see from the proof of Lemma \ref{lem:iff con} that  $E+F$ has empty interior.
\end{exa}

{\noindent \bf  Acknowledgements}. This work was a part of the author’s PhD thesis. The author is grateful to De-Jun Feng for valuable discussions and suggestions.

\end{document}